\documentclass[12pt]{article}
\usepackage{amssymb,latexsym,amscd,amsmath,amsfonts,amsthm,enumerate}

\newtheorem{theorem}{Theorem}[section]
\newtheorem{lemma}[theorem]{Lemma}
\newtheorem{proposition}[theorem]{Proposition}
\newtheorem{corollary}[theorem]{Corollary}

\newtheoremstyle{definition}
  {6pt}
  {6pt}
  {}
  {}
  {\bfseries}
  {.}
  {.5em}
  {}%

\theoremstyle{definition}
\newtheorem{definition}[theorem]{Definition}

\newtheoremstyle{remark}
  {6pt}
  {6pt}
  {}
  {}
  {\bfseries}
  {.}
  {.5em}
  {}%

\theoremstyle{remark}
\newtheorem{remark}[theorem]{Remark}

\makeatletter
\renewcommand\@makefntext[1]{%
\setlength\parindent{1em}%
\noindent

\makebox[1.8em][r]{}{#1}} \makeatother

\begin{document}
\title{\bf \large  THE DEGREE SEQUENCE  OF IDEALS \\AND MULTIPLICITIES}
\author{
 \centerline{Duong Quoc Viet}\\
\small Department of Mathematics, Hanoi National University of Education\\
\small 136 Xuan Thuy street, Hanoi, Vietnam\\
\small  duongquocviet@fmail.vnn.vn\\
}

   \date{}
\maketitle \centerline{\parbox[c]{11.5 cm}{ \small  ABSTRACT: This
paper investigates the relationship between
 multiplicities and the degree sequence of ideals in graded
algebras, gives  multiplicity equations of graded rings via the
degree sequence of ideals, and characterizes mixed multiplicities
and multiplicities of Rees rings in terms of the degree sequence
of ideals. As an application, the paper answered an open question
on the multiplicity of Rees rings.
 }}

\section{Introduction}
Establishment of formulas for calculating multiplicities is always
an interesting problem and attracted much attention. Let
$S=\bigoplus_{n\ge 0} S_n$ be a finitely generated standard
$\mathbb{N}$-graded algebra over an infinite field  $k.$  Denote
by $\frak m$ the maximal homogeneous ideal of $S.$
 It is well-known that if $c_1 \le
\cdots \le c_t$  are the degrees of the elements of an arbitrary
homogeneous minimal basis of a homogeneous ideal $I$ then the
sequence $(c_1,\ldots, c_t)$ does not depend upon the choice of
the homogeneous minimal basis, and this sequence is called the
{\it degree sequence} of $I.$

\vskip.2cm Let $I$ be an ideal of $S.$  One always want to know
the relationship between
 multiplicities of objects arising from $I$ and the degree
sequence of $I.$ \footnotetext{\begin{itemize} \item[ ] This
research was in part supported by a grant from  NAFOSTED. \item[
]{\it Mathematics Subject  Classification}(2010): Primary 13H15.
Secondary 14C17, 13D40, 13C15. \item[ ]{\it Key words and
phrases}: The degree sequence, multiplicity, mixed multiplicity,
(FC)-sequence.
\end{itemize}}

\vskip.2cm In fact,  Herzog, Trung, Ulrich in 1992 \cite {HTU}
started the investigation of the multiplicity of blow-up rings of
ideals generated by $d$-sequences in terms of the degree sequence.

\vskip.2cmWe would  like now to begin    with  mixed
multiplicities and multiplicities of Rees rings $R(I)$ of
homogeneous ideals $I$ generated by a subsystems of parameters.
This problem  originated from a paper of Trung in 1993
\cite[Theorem 3.3]{Tr} by the following.

\begin{theorem}[$\mathrm{\cite {Tr}}$] \label{th0.2} Let $I$ be a
homogeneous ideal generated by a subsystem of parameters $x_1,
\ldots,x_h$ which is an $I$-filter-regular sequence with $\deg x_1
\le \cdots \le \deg x_h.$ Denote by $(a_1,\ldots,a_h)$ the degree
sequence of $I.$ Then $e(R(I)) = (1+\sum_{i=1}^{h-1} a_1 \cdots
a_i) e(S).$
\end{theorem}
By using weak-(FC)-sequences in \cite{Vi}(see e.g. \cite{DMT, DV,
VT}),  Viet in 2003 \cite{ViS} obtained formulas for mixed
multiplicities and multiplicities of Rees rings of arbitrary
homogeneous equimultiple ideals in $S,$ i.e, the additional
condition on the generation of $I$ in Theorem \ref {th0.2} has
been  dropped.  We emphasize
 that the proof of \cite{ViS} is based on the existence of
homogeneous weak-(FC)-elements in $I$ as in  Lemma \ref{lm1.6}.
However, this is not correct (see Remark \ref {rm36b}). So the
results in \cite{ViS} are  not yet proven. And  \cite{TV2} in 2010
gave
 the following open question which has also been stated in \cite{Tr} in 1993.
 \vskip.2cm \noindent{\bf Question 1} \cite[Question
7.7] {TV2}{\bf.} {\it Can one drop the condition $\deg x_1 \le
\cdots \le \deg x_h$ in Theorem \ref{th0.2}}?

  \vskip.2cmIt has long been known that the approach, which based
on the existence of sequences of homogeneous elements, encountered
obstructions in the progress of calculating multiplicities.
 This is one of motivations
to help us giving another approach for this problem. Our approach
is started by the construction of the following object.

  \vskip.2cmNote that for any non-zero element $ x$ of $S,$ one can
always write uniquely in the form  $x = x_{m_1} + \cdots+x_{m_t},$
here $m_1 < \cdots < m_t$ and $ 0 \ne x_{m_i} \in S_{m_i}$ for all
$1 \le i \le t,$ then  we put $\mathrm{in}x = x_{m_1}$ and $o(x) =
m_1.$   Let $I$ be an ideal of $S.$ Denote by $\mathrm{in}I$ the
ideal generated by  $\{\mathrm{in}x \mid x\in I \setminus \frak mI
\}.$ Then the degree sequence of $\mathrm{in}I$ is called the {\it
degree sequence} of $I.$ Proposition \ref{lm1.1v} showed the
existence of  minimal bases $x_1, \ldots, x_n$ of $I$ such that
  $(o(x_1),\ldots,o(x_n))$ is the degree sequence of $I.$

\vskip.2cm This paper examines the relationship between
multiplicities of ideals in graded algebras and the degree
sequence of ideals, hereby answers Question 1.
 And as one might expect,  we first obtain the following
result for the Hilbert-Samuel multiplicity of graded domains, in
terms of the degree sequence of ideals. \vskip 0.2cm \noindent
  {\bf Theorem 1.2} (Theorem \ref{th1.1a}){\bf.} {\it Let $S=\bigoplus_{n\ge 0} S_n$
  be  a finitely generated standard
$\mathbb{N}$-graded algebra over an infinite field  $k$ of $\dim S
= d >0.$ Let $I$ be an $\frak m$-primary ideal of $S$. Assume that
$S$ is a domain and  $(c_1,\ldots,c_{d})$ is the degree sequence
of a
  minimal reduction of $I.$
 Then we have $e(I; S) = c_1\cdots c_d e(S).$}

 \vskip.2cmIt should be noted that this theorem does not hold in general if
one omits the condition that $S$  is a domain (see Remark
\ref{rm35}).

 \vskip.2cmHowever,  using this theorem as a tool  we prove the following
theorem.

 \vskip 0.2cm \noindent{\bf Theorem 1.3} (Theorem \ref
{th1.1v}){\bf.} {\it Let $S=\bigoplus_{n\ge 0} S_n$ be  a finitely
generated standard $\mathbb{N}$-graded algebra over an infinite
field  $k$ of $\dim S = d >0.$ Let $x_1, \ldots, x_d$ be a system
of parameters of $S.$ Assume that $\mathrm{in}x_1, \ldots,
\mathrm{in}x_d$ is a system of parameters of $S.$ Then we have}
$e(x_1,\ldots, x_d; S)= o(x_1)\cdots o(x_d) e(S).$ \vskip.2cm
Theorem 1.3 seems to become an effective support for
characterizing multiplicities in terms of the degree sequence of
ideals. In fact, Theorem 1.3 yielded  the interesting applications
for mixed multiplicities (see Theorem \ref{th1.3} and Corollary
\ref{th3.6a}).

 \vskip.2cmFurther, by using Theorem 1.3 and Lemma \ref{lm1.6} we immediately
answered to the Question 1 by the following result.

\vskip 0.2cm \noindent
  {\bf Corollary 1.4} (Corollary  \ref {th3.6}){\bf.}
  {\it Let $I$ be a homogeneous  equimultiple
ideal in a   standard graded algebra $S$  of $\mathrm{ht}I =h >
0.$ Let $(a_1, \ldots,a_h)$ be the degree sequence of a
homogeneous minimal reduction of $I.$ Set $\dim S =d.$  Denote by
$e({\frak m}^{[d-i]},I^{[i]} ; S)$ the mixed multiplicity of
$(I,\frak m)$
 of the type $(i, d-i)$ for $ i \le d-1.$ Then we have:
\begin{enumerate}[\rm (i)]
 \item $ e({\frak m}^{[d-i]},I^{[i]} ; S) = a_1 \cdots a_ie(S)$  for all $ i \le h-1.$
 \item $e(R(I)) =
(1+\sum_{i=1}^{h-1}a_1 \cdots a_i)e(S).$
\end{enumerate} }
\enlargethispage{1.0 cm}

Corollary 1.4 is also Theorem 3.1 given by Viet in \cite{ViS}. So
by another approach, the present paper showed that \cite [Theorem
3.1]{ViS} and hence the other results on mixed multiplicities and
multiplicities of Rees rings in \cite{ViS} are true, and the
Question 1 has a positive answer (see Remark \ref{rm36b}).

 \vskip.2cmNotice that Corollary 1.4 is a particular case of Corollary
\ref{th3.6a}. And although in general there does not exist a
homogeneous weak-(FC)-sequence in $I$ as in  Lemma \ref{lm1.6},
this lemma  showed the existence of weak-(FC)-sequences in $I$
having  the same degree sequence as a homogeneous minimal
reduction of $I.$ This is  an important property used for the
proof of Corollary 1.4.
 The results of the paper showed the
important role of the degree sequence in studying  multiplicities
 of ideals.

 \vskip.2cmThis paper is divided into  three  sections. Section 2 is devoted
to the discussion of the degree sequence of ideals (Proposition
\ref{lm1.01}, Proposition \ref{lm1.1v}) and multiplicities of
graded algebras (Proposition \ref{lm1.1}, Corollary \ref{th1.1h},
Corollary \ref{th1.1q}, Theorem \ref{th1.1a}, Theorem \ref
{th1.1v}). And as applications, we get results for mixed
multiplicities and multiplicities of Rees rings of ideals in
graded rings (Theorem \ref{th1.3}, Corollary \ref{th3.6a}).
Section 3 investigates the relationship between the degree
sequence of homogeneous equimultiple ideals and the degree
 sequence of weak-(FC)-sequences of these ideals (Lemma
\ref{lm1.6}) that will be used in the proofs for results on
 mixed multiplicities and multiplicities of Rees rings
of homogeneous  equimultiple  ideals (Corollary \ref{th3.6}, and
Corollary \ref{co37c}). And Question 1 is answered by Corollary
\ref{th3.6}.

\section{ Multiplicities of  Graded  Rings}

In this section,  we first define some objects in graded rings,
and study the relationship between
   multiplicities; mixed multiplicities of graded algebras and
   the degree sequence of ideals,
 give  multiplicity formulas of
  graded rings; mixed multiplicity formulas and the multiplicity of Rees rings
 via  the degree sequence of
ideals.

\vskip.2cmOur approach is started by the construction of the
following useful objects  in standard $\mathbb{N}$-graded
algebras.

\vskip.2cm
 Recall that
$S=\bigoplus_{n\ge 0} S_n$ is  the finitely generated standard
$\mathbb{N}$-graded algebra over an infinite field $k,$ i.e., $S$
is generated over $k$ by elements of total degree $1,$ and $\frak
m$ the maximal homogeneous ideal of $S.$ Then  for any element  $
x \ne 0$ of $S,$ one can always write uniquely in the form  $x =
x_{m_1} + \cdots+x_{m_t},$ here $m_1 < \cdots < m_t$ and $ 0 \ne
x_{m_i} \in S_{m_i}$ for all $1 \le i \le t.$ In this case, we put
 $o(x)= {m_1} = \deg x_{m_1}$ and $\mathrm{in}x = x_{m_1}.$
 Assign
$\mathrm{in}0 = 0;$ $o(0) = + \infty.$  Set $o(U) = \min \{o(x)
\mid x \in U\}$ for any $U  \subset S.$

\vskip.2cm Let  $b_1, \ldots, b_n \in S_1$ be a minimal basis of
$\frak m$ and fix this order.
  Set $\mathbf{m} = (m_1,\ldots, m_n); $ $|\mathbf{m}| = m_1+\cdots+ m_n$ and
$\mathbf{b}^\mathbf{m} = b_1^{m_1}\cdots b_n^{m_n}$ for any
$(m_1,\ldots, m_n) \in \mathbb{N}^n.$  It is obvious that $ S =
\sum_{\mathbf{m} \in \mathbb{N}^n}k\mathbf{b}^\mathbf{m}$. And it
is a plain fact that $C_u =\{\mathbf{b}^\mathbf{m} \mid |
\mathbf{m}| = u \}$ generates $\frak m^u$  for any $u > 0.$ Then
for any $u > 0,$ one can choose a minimal basis $B_u \subset C_u $
of $\frak m^u,$ and fix this basis. Set $B_0 = \{1\}.$
   Then $$S =
\bigoplus_{u \ge 0}(\bigoplus_{\mathbf{b}^\mathbf{m} \in
B_u}k\mathbf{b}^\mathbf{m}) \;\mathrm{and}\; \frak m =
\bigoplus_{u
> 0 }(\bigoplus_{\mathbf{b}^\mathbf{m} \in
B_u}k\mathbf{b}^\mathbf{m}).$$

 \vskip.2cmNow we define an order on $\mathbb{N}^{n}$ as follows: $\mathbf{m}
<\mathbf{m}'$ if the first non-zero component from the left side
of $(|\mathbf{m}| - |\mathbf{m}'|, m_1 - m'_1, \ldots, m_n -
m'_n)$ is negative.

 \vskip.2cm Note that each
non-zero element  $x$ of $S$ can always  be written uniquely in
the form $x = \sum_{\mathbf{b}^\mathbf{m} \in \bigcup_{u\ge0}B_u}
x_{\mathbf{m}}\mathbf{b}^\mathbf{m},$ where $x_\mathbf{m} \in k$
and there are only finitely many $\mathbf{m}$ such that $
x_{\mathbf{m}} \ne 0,$ in this case, set $x^* =
\mathbf{b}^\mathbf{v}$ if $\mathbf{v} = \min\{\mathbf{m}\mid
x_\mathbf{m} \not=0\}.$  Assign $0^* = 0.$   Denote by $I^*$ the
ideal generated by $\{x^* \mid x\in I \setminus \frak mI \}$ for
any ideal $I$ of  $S.$

 \vskip.2cmFor   any  homogeneous ideal $I$ of $S,$ it is well-known that if
$c_1 \le \cdots \le c_t$  are  the degrees of the elements of an
arbitrary homogeneous minimal basis of $I,$ then the sequence
$(c_1,\ldots, c_t)$ does not depend on the choice of the minimal
basis, and this sequence is called the {\it degree sequence} of
$I.$ The case of an arbitrary ideal $I$ of $S,$ denote by
$\mathrm{in}I$ the ideal generated by  $\{\mathrm{in}x \mid x\in I
\setminus \frak mI \}.$ Then the degree sequence of $\mathrm{in}I$
is called the {\it degree sequence} of $I.$ Denote by $\mu(I)$ the
minimal number of generators of $I.$
\begin{remark} \label{rm1.2s} If $I$ is a homogeneous ideal, then $I$ has
homogeneous minimal bases. So it is easy to check that
$\mathrm{in}I = I.$ Hence the above notion of the degree sequence
is an extension of the ordinary notion of the degree sequence for
homogeneous ideals.
\end{remark}

\vskip.2cm Further, the relationship between $I;$  $\mathrm{in}I$
and $I^*$ is shown by the following lemma.

\begin{proposition}\label{lm1.01} Let $I$ be an ideal of $S.$
We have the following statements.
\begin{enumerate}[\rm (i)]
 \item $\mu(I^*) = \mu(I) = \mu(\mathrm{in}I).$
\item There exists a minimal basis $x_1, \ldots, x_n$ of $I$ such
that $\mathrm{in}x_1,  \ldots, \mathrm{in}x_n$  and $x^*_1,
\ldots, x^*_n$ are minimal bases of $\mathrm{in}I$ and $I^*,$
respectively.

\item The degree sequences of $I$ and $I^*$ are the same.
  \end{enumerate}
\end{proposition}
\begin{proof} Now assume that $\mu(I) =n.$  Then there exist
 $x_1, \ldots, x_n \in I \setminus \frak mI $ such
that $x_1, \ldots, x_n$ is a minimal basis of $I.$ And
  without loss of generality, we can assume
that $\mathbf{v}_1 \le \cdots \le \mathbf{v}_d,$ here   $x_i^* =
\mathbf{b}^{\mathbf{v}_i}$ for all $1 \le i \le d.$ Notice that if
$x^*_i = x^*_{i+1}$ and $$x_i = ax^*_i+\cdots; \;\;x_{i+1} =
bx^*_{i+1}+\cdots$$ for $0 \ne a, b \in k,$ then one can replace
$x_{i+1}$ by $y_{i+1}= x_{i+1}- ba^{-1}x_i$ for any $1 \le i < n.$
Then we get a minimal basis   $y_1, \ldots,y_{n}$ of $I$ such that
$y^*_1, \ldots,y^*_{n}$ are pairwise distinct. Consequently, we
can assume that
 $x_1, \ldots, x_n$ is a minimal basis of $I$
 with $x^*_i \ne x^*_j $ for all $1 \le i \ne j \le n.$
 Set $X = \{x^*_1,
\ldots, x^*_n\}.$   Let  $x\in I \setminus \frak mI.$ Since $x_1,
\ldots, x_n$ is a minimal basis of $I,$ it follows that $x=\sum_{i
=1}^na_ix_i + y,$ here  $a_1, \ldots, a_n \in k$  and $y\in \frak
mI.$ Since $y\in \frak mI,$ it follows that $y^* \notin X \cup
\{x^*\}.$ Because $x^*_1, \ldots,x^*_{n}$ are pairwise distinct,
we get $x^* = x^*_i$ for some $1 \le i \le n.$ Hence $x^*  \in X.$
So $X=\{x^* \mid x\in I \setminus \frak mI \}.$ Therefore  $I^*$
is generated by $X.$ Now since $x^*_1, \ldots,x^*_{n}$ are
pairwise distinct, it follows that $x^*_1, \ldots,x^*_{n}$ are
$k$-linearly independent.
 Therefore $\mu(I^*) = n,$ and $x^*_1,
\ldots,x^*_{n}$ is a minimal basis of $I^*.$ Next, set
$\mathrm{in}x_i = z_i$ for all $1 \le i \le n.$ Then we have
$x^*_i = z^*_i$ for all $1 \le i \le n.$ So $az_i \ne bz_j$ for
all $ 0 \ne a, b \in k$ and $1 \le i \ne j \le n.$

Set $Y = \{
 x \ne 0 \mid  x = a z_j; \; a \in k; \; 1 \le j \le n \}.$
 Then for any  $x\in I \setminus \frak mI,$ since $x_1,
\ldots, x_n$ is a minimal basis of $I,$ we get $x=\sum_{i
=1}^na_ix_i + y,$ here  $a_1, \ldots, a_n \in k$  and $y\in \frak
mI.$ Since $y\in \frak mI,$ it follows that $\mathrm{in}y \notin Y
\cup \{a\mathrm{in}x\}$ for all $0 \ne a \in k.$ Then since $az_i
\ne bz_j$ for all $ 0 \ne a, b \in k$ and $1 \le i \ne j \le n,$
we obtain $\mathrm{in}x = az_i$ for some $1 \le i \le n$ and $ a
\in k.$ So $\mathrm{in}x \in Y.$ We get  $Y=\{\mathrm{in}x \mid
x\in I \setminus \frak mI \}.$ Hence $\mathrm{in}I = (z_1, \ldots,
z_n).$
 Recall that $az_i
\ne bz_j$ for all $ 0 \ne a, b \in k$ and $1 \le i \ne j \le n,$
it follows that
  $z_1, \ldots, z_n$ are $k$-linearly
independent.
 Thus $\mu(\mathrm{in}I) = n$ and
$\mathrm{in}x_1, \ldots, \mathrm{in}x_n$ is a minimal basis of
$\mathrm{in}I.$ We get (i) and (ii). By  (ii) we immediately
obtain (iii).
 \end{proof}
\begin{remark} \label{rm1.2u} From the proof of Proposition
\ref{lm1.01}, it follows that  $X=\{x^* \mid x\in I \setminus
\frak mI \}$ is a minimal basis of  $I^*$  and $\mu (I) = |X|$ for
any ideal $I$ of $S.$ Moreover, if $x_1, \ldots, x_n$ is a minimal
basis of $I$ such that  $x^*_1, \ldots,x^*_{n}$ are pairwise
distinct, then $\mathrm{in}x_1, \ldots, \mathrm{in}x_n$  and
$x^*_1, \ldots, x^*_n$ are minimal bases of $\mathrm{in}I$ and
$I^*,$ respectively. Since $o(x_i) = \deg \mathrm{in}x_i$ for all
$1 \le i \le n,$ hence the degree of $I$ is a permutation of
$o(x_1),\ldots, o(x_n).$ In fact, as the proof of Proposition
\ref{lm1.01}, one also follow that if $x_1, \ldots, x_n$ is a
minimal basis of $I$ such that $a\mathrm{in}x_i \ne
b\mathrm{in}x_j$ for all $ 0 \ne a, b \in k$ and $1 \le i \ne j
\le n,$ then $\mathrm{in}x_1, \ldots, \mathrm{in}x_n$ is a minimal
basis of $\mathrm{in}I.$ And in this case, the degree sequence  of
$I$ is a permutation of $o(x_1),\ldots, o(x_n).$
\end{remark}

Remark \ref {rm1.2u} yields the following conclusion.

\begin{proposition} \label{lm1.1v} Let $x_1, \ldots, x_n$ be a minimal
basis of $I$ with $o(x_1) \le \cdots \le o(x_n)$ and $x^*_1,
\ldots,x^*_{n}$ pairwise distinct.
  Set $o(x_i) = c_i$ for all $1 \le i \le n.$ Then
the sequence $(c_1,\ldots, c_n)$ does not depend on the choice of
this minimal basis, and this sequence is  the  degree sequence of
$I.$

\end{proposition}

 \vskip.2cmThe following result is
an important tool  of this paper.

\begin{proposition} \label{lm1.1} Let $I$ be an ideal of $S.$
Assume that $S$ is a domain.
 Then we have
$$e(S/I)=  e(S/\mathrm{in}I).$$
\end{proposition}
\begin{proof} By Proposition \ref{lm1.01},
there exists a minimal basis $x_1, \ldots, x_n$ of $I$ such that
 $\mathrm{in}x_1, \ldots, \mathrm{in}x_n$ is a  minimal basis  of $\mathrm{in}I.$
Let $x$ be an element of $I.$ Then $x= \sum_{i=1}^ny_ix_i$ for
certain  elements  $y_1, \ldots, y_n \in S.$  Since $S$ is a
domain, it follows that $\mathrm{in}x=
\mathrm{in}\sum_{i=1}^ny_i\mathrm{in}x_i.$ Hence $ \mathrm{in}x
\in \mathrm{in}I$ for all $x \in I.$ Any $z=
\sum_{i=1}^nz_i\mathrm{in}x_i \in \mathrm{in}I$ for $z_1, \ldots,
z_n \in S,$  we have $$\mathrm{in}z=
\mathrm{in}\sum_{i=1}^nz_i\mathrm{in}x_i  =
\mathrm{in}\sum_{i=1}^nz_ix_i$$ since $S$ is a domain. Note that
$u=\sum_{i=1}^nz_ix_i \in I.$ So $\mathrm{in}z = \mathrm{in}u$ for
some $u \in I.$ From this it follows that $$\{\mathrm{in}x \mid x
\in I\} = \{\mathrm{in}x \mid x \in \mathrm{in}I\}.$$ Hence we
have
$$\{\mathrm{in}x \mid x \in I\}\bigcap S_m + \bigoplus_{n \ge m +1
}S_n = \{\mathrm{in}x \mid x \in \mathrm{in}I\}\bigcap S_m +
\bigoplus_{n \ge m +1 }S_n$$ for all $m \in \mathbb{N}.$ Moreover,
it is plain that
$$\{\mathrm{in}x \mid x \in I\}\bigcap S_m + \bigoplus_{n \ge m +1
}S_n = I\bigcap [\bigoplus_{n \ge m }S_n] + \bigoplus_{n \ge m +1
} S_n $$ and $ \{\mathrm{in}x \mid x \in \mathrm{in}I\}\bigcap S_m
+ \bigoplus_{n \ge m +1 }S_n = \mathrm{in}I \bigcap[\bigoplus_{n
\ge m }S_n] + \bigoplus_{n \ge m +1 }S_n.$ So $$\bigoplus_{n \ge m
+1 }S_n+I \bigcap[\bigoplus_{n \ge m }S_n] = \bigoplus_{n \ge m +1
}S_n+ \mathrm{in}I \bigcap[\bigoplus_{n \ge m }S_n]$$ for all $m
\in \mathbb{N}.$  Now since $ {\frak m }^m = \bigoplus_{n \ge m
}S_n,$   we obtain
\begin{equation}\label{eq1} \frak m^{m+1}+I \bigcap\frak m^{m} =
\frak m^{m+1}+ \mathrm{in}I \bigcap\frak m^{m}\end{equation} for
all $m \in \mathbb{N}.$
 Denote by $F= \{{\frak m}^m\}_{m \ge 0}$ the ${\frak
m}$-adic filtration of $S,$ and denote by $$gr_F(S/I) =
\bigoplus_{m \ge {0}} (\frak m^m(S/I)/\frak m^{m+1}(S/I))$$ the
associated graded module of the $S$-module $S/I$ with respect to
the filtration $F.$ It can be verified that
$$gr_F(S/I)\cong \bigoplus_{m \ge {0}} (\frak
m^m+I/\frak m^{m+1}+I)\cong \bigoplus_{m \ge {0}}(\frak m^m/(\frak
m^{m+1}+I \cap\frak m^{m})).$$   Hence by (\ref{eq1})  we get
$$gr_F(S/I) \cong gr_F(S/\mathrm{in}I).$$ Note that
$e(S/I)= e(gr_F(S/I)),$ thus $e(S/I)= e(S/\mathrm{in}I). $
\end{proof}

 \vskip.2cmThe proof  of  Proposition \ref{lm1.1} yields  the following
comment.

 \vskip.2cm\begin{remark} \label{rm1.2c} Let $S$ be a domain. Recall that  we have
$$\{\mathrm{in}x \mid x \in E\} = \{\mathrm{in}x \mid x \in
\mathrm{in}E\}$$ for any ideal $E$ of $S$ by the proof of
Proposition \ref{lm1.1}. And it is a simple master to give
examples that this equation is not true in general if $S$ is not a
domain. Note that  $\mathrm{in}E$ is a homogeneous ideal, we
obtain $\{\mathrm{in}x \mid x \in \mathrm{in}E\} \subset
\mathrm{in}E.$ So $\{\mathrm{in}x \mid x \in E\} \subset
\mathrm{in}E.$ Next, we would like to give an example to show
that this inclusion is not true in general if $S$ is not a domain.

Let  $S = k[X, Y]/(XY, X^2),$ here $k$ is a field and
 $X$ and  $Y$ are variables. Then $S$ is not a domain. Denote by $\bar X$ and  $\bar Y$ the
 images of $X$ and  $Y$ in $S,$ respectively.
Let $x = \bar X + \bar Y^2$ be an element of $S.$  Set $E = (x).$
Then $\mathrm{in}E = (\bar X)$  and $y= x^2 = \bar Y^4 \in E.$
However $\mathrm{in}y = \bar Y^4 \notin \mathrm{in}E.$

 \vskip.2cmThe following corollaries of Proposition \ref{lm1.1} are  a pivotal connection
 for the approach of the paper, and  which
are proven to be useful in  this paper.

 \vskip.2cmLet $I$ be an $\frak m$-primary ideal, and let $S$ be a domain. Then $\frak m^u \subset I$
for a certain integer $u
> 0.$  From this it follows that $\{\mathrm{in}x \mid x
\in \frak m^u\} \subset \{\mathrm{in}x \mid x \in I\} \subset
\mathrm{in}I$ by  Remark \ref {rm1.2c}. So we get $\mathrm{in}
\frak m^u \subset \mathrm{in}I.$  Since $\frak m^u$ is a
homogeneous ideal, $\mathrm{in}\frak m^u = \frak m^u$ by Remark
\ref{rm1.2s}. Hence $\sqrt{\mathrm{in}\frak m^u}  = \frak m.$
Consequently $\mathrm{in}I$ is an $\frak m$-primary ideal of $S.$
 By Proposition \ref {lm1.1}, we have $$e(S/I)=
e(S/\mathrm{in}I).$$ Note that $I$ and $\mathrm{in}I$   are $\frak
m$-primary ideals, it follows that $S/I$  and $S/\mathrm{in}I$ are
Artinian rings. Hence $e(S/I)= \ell(S/I)$ and  $e(S/\mathrm{in}I)
= \ell(S/\mathrm{in}I).$ Thus $$\ell(S/I)= \ell(S/\mathrm{in}I).$$

Hence we get the following result.
\begin{corollary}\label{th1.1h}  Let $I$ be an $\frak m$-primary ideal of $S$. Assume that
$S$ is a domain. Then $\mathrm{in}I$ is an $\frak m$-primary ideal
of $S$ and $\ell(S/I)= \ell(S/\mathrm{in}I).$
 \end{corollary}

Now assume that $S$ is a domain with $\dim S = d
>0$ and  $q$ is a parameter ideal of $S.$ By Proposition \ref{lm1.01},
there exists a minimal basis $x_1, \ldots, x_d$ of $q$ such that
 $\mathrm{in}x_1, \ldots, \mathrm{in}x_d$ is a  minimal basis  of $\mathrm{in}q.$
Note that $\mathrm{in}q$ is an $\frak m$-primary ideal of $S$ by
Corollary \ref{th1.1h}.
  So $x_1, \ldots, x_d$ and  $\mathrm{in}x_1, \ldots,
\mathrm{in}x_d$  are systems of parameters of $S.$ Hence these
systems are algebraically independent over $k$ (see e.g. \cite
[Corollary 11.21]{A}). Therefore it can be verified that  $$
\{\mathrm{in}x_1^{n_1}\cdots\mathrm{in}x_d^{n_d} =
\mathrm{in}[x_1^{n_1}\cdots x_d^{n_d}] \mid n_1+\cdots+n_d =n \}$$
is a minimal basis of $\mathrm{in}q^n.$ So $\mathrm{in}q^n =
(\mathrm{in}x_1, \ldots, \mathrm{in}x_d)^n = (\mathrm{in}q)^n$ for
all $n >0.$ Then by Corollary \ref{th1.1h}, we get
  $$\ell(S/q^n)=
\ell(S/(\mathrm{in}q)^n)$$ for all $n
>0.$ Consequently
$$ e(q; S)= e(\mathrm{in}q; S).$$

These comments yield:
\begin{corollary}\label{th1.1q}   Assume that
$S$ is a domain and   $q$ is a parameter ideal of $S.$ Then
$\mathrm{in}q$ is a parameter ideal of $S$ and $ e(q; S)=
e(\mathrm{in}q; S).$
 \end{corollary}

\end{remark}

Let
 $I$ be an  ideal of   a  Noetherian   local ring  $(A, \frak n)$  with the maximal
ideal $\frak{n}.$  An ideal $J$ is called a {\it reduction} of an
ideal $I$ if $J \subseteq I$ and $I^{n+1} = JI^n$ for some $n.$ A
reduction $J$ is called a {\it minimal reduction } of $I$ if it
does not properly contain any other reduction of $I$ \cite{NR}.
Set $\ell(I) = \dim\underset{n\ge0}\bigoplus (I^n/{\frak n}I^n).$
Then one called $\ell(I)$  the {\it analytic spread } of $I.$ If
the residue field $k = A/{\frak n}$ is infinite, then the minimal
number of generators of every minimal reduction of $I$ is equal to
$\ell(I)$ \cite{NR}. It is well known that $\mathrm{ht}I \le
\ell(I) \le \dim A,$ and $I$ is called an {\it equimultiple ideal
} if $\ell(I) = \mathrm{ht}I.$

\vskip.2cm The following result will play a crucial role in the
approach  of this paper.

\begin{theorem}\label{th1.1a} Let $S=\bigoplus_{n\ge 0} S_n$ be  a finitely generated standard
$\mathbb{N}$-graded algebra over an infinite field  $k$ of $\dim S
= d >0.$ Let $I$ be an $\frak m$-primary ideal of $S$. Assume that
$S$ is a domain and
  $(c_1,\ldots,c_{d})$ is the degree sequence of a
  minimal reduction of $I.$
 Then we have $e(I; S) = c_1\cdots c_d e(S).$
  \end{theorem}
\begin{proof} Let $J$ be a
  minimal reduction of $I$ of the degree sequence
  $(c_1,\ldots,c_{d}).$ Note that since $I$ is  $\frak
m$-primary, it follows that $\ell(I) = \mathrm{ht}I = d.$ Hence
$\mu(J)= d$ and $J$ is a parameter ideal of $S.$
  Recall that by Corollary \ref{th1.1q}, we have $ e(J; S)=e(\mathrm{in}J; S).$
Since $J$ is a   minimal reduction of $I,$ we obtain $e(I; S)=
e(J; S)$ by \cite[Theorem 1]{NR}.   By Proposition \ref{lm1.01},
there exists a minimal basis $x_1, \ldots, x_d$ of $J$ such that
 $\mathrm{in}x_1, \ldots, \mathrm{in}x_d$ is a minimal basis of
$\mathrm{in}J.$ Set $y_i = \mathrm{in}x_i$ for all $1 \le i \le
d.$ Then $y_1,\ldots, y_d$ is a system of parameters for $S,$ and
one can consider the degree sequence of $J$ as
$$(\deg y_1, \ldots, \deg y_d) = (c_1,\ldots,c_{d}).$$  Next
we recall the following proof in \cite [Lemma 3.3]{VT0} for $
e(\mathrm{in}J; S) = c_1\cdots c_d e(S).$

Put $c = c_1\cdots c_d$ and $u_i = c/c_i$ for all $1 \le i \le d.$
 Then we have
 $\deg y_i^{u_i} = c$ for all $1 \le i \le d.$ So
 $$ \{{y_1}^{u_1},\ldots, {y_d}^{u_d}\} \subset S_c.$$
From the above facts, we get $$ c^de(S)= e(({y_1}^{u_1},\ldots,
{y_d}^{u_d}); S)= u_1\cdots u_d e((y_1, \ldots, y_d); S) =
u_1\cdots u_d e(\mathrm{in}J; S).$$ Consequently,  $
e(\mathrm{in}J; S) = c_1\cdots c_d e(S).$ We get $e(I; S) =
c_1\cdots c_d e(S).$
\end{proof}

Notice  that when $E$ is a homogeneous ideal generated by a
subsystem of parameters  of $S$ and $(a_1,\ldots, a_h)$ is the
degree sequence of $E,$ upon special computations, Trung in
\cite{Tr} proved that
$$e(G(E)) = a_1\cdots a_h e(S),$$ here $G(E)=\bigoplus_{n\ge  0}(E^n/E^{n+1})$
 is the associated graded ring of $E.$
Using this result, one can also get $$ e(\mathrm{in}J; S) =
c_1\cdots c_d e(S)$$ since $e(G(\mathrm{in}J)) = e(\mathrm{in}J;
S).$

\begin{remark}\label{rm35}
From Theorem \ref{th1.1a}, one may raise a question:
 Does the theorem hold if $S$ is not a domain?
 Consider the case $S = k[X, Y]/(XY, X^2),$ here $k$ is a field and
 $X$ and  $Y$ are variables (see Remark \ref{rm1.2c}). Denote by $\bar X$ and  $\bar Y$ the
 images of $X$ and  $Y$ in $S,$ respectively.
Then it is easily seen that $\dim S= 1;$ $e(S) = 1$ and $x = \bar
X + \bar Y^2$ is a system of parameters for $S;$  $x^2 = \bar
Y^{4}$ and $o(x) =1.$  Now if Theorem \ref{th1.1a} is true for
$S,$ then $e(x; S) = e(S).$ So $e(x^2; S) = 2e(S).$ Because
$e(x^2; S) = e(\bar Y^{4};S) = 4e(\bar Y;S),$  $e(S)=2e(\bar
Y;S).$  Hence  $e(S) \ne 1.$
 This example shows that
Theorem \ref{th1.1a} does not hold in general if one omits  the
assumption that $S$ is a domain.
\end{remark}

\vskip.2cm However, using Theorem \ref {th1.1a} we prove the
following theorem for multiplicities of arbitrary
$\mathbb{N}$-graded algebras.

\begin{theorem}\label{th1.1v} Let $S=\bigoplus_{n\ge 0} S_n$ be  a finitely generated standard
$\mathbb{N}$-graded algebra over an infinite field  $k$ of $\dim S
= d >0.$ Let $x_1, \ldots, x_d$  be a system of parameters of $S.$
Assume that   $\mathrm{in}x_1, \ldots, \mathrm{in}x_d$ is a system
of parameters of $S.$ Then we have
$$e(x_1,\ldots,
x_d; S)= o(x_1)\cdots o(x_d) e(S).$$
\end{theorem}
\begin{proof}
Denote by $\Lambda$ the set of minimal prime ideals $\mathfrak{p}$
of $S$ such that $\dim S/\mathfrak{p}=d.$ For any $\mathfrak{p}
\in \Lambda,$ denote by ${\bar x}_1, \ldots, {\bar x}_d$ and
$\overline{\mathrm{in}x}_1, \ldots, \overline{\mathrm{in}x}_d$ the
images of $x_1, \ldots, x_d$ and $\mathrm{in}x_1, \ldots,
\mathrm{in}x_d$ in $S/\mathfrak{p},$ respectively. Since
$\mathfrak{p}$ is a minimal prime ideal with $\dim S/\mathfrak{p}
= d,$  it follows that $\mathfrak{p}$ is a homogeneous ideal and
 ${\bar x}_1, \ldots, {\bar
x}_d$ and $\overline{\mathrm{in}x}_1, \ldots,
\overline{\mathrm{in}x}_d$ are systems of parameters of the graded
algebra $S/\mathfrak{p}.$ Note that $\overline{\mathrm{in}x}_i \ne
0$ for all $1 \le i \le d,$ we get $\overline{\mathrm{in}x}_i
=\mathrm{in}{\bar x}_i$ for all $1 \le i \le d.$ Since
$\mathrm{in}{\bar x}_1, \ldots, \mathrm{in}{\bar x}_d$ is a system
of parameters of $S/\mathfrak{p},$ it follows that
$a\mathrm{in}{\bar x}_i \ne b\mathrm{in}{\bar x}_j$ for all $ 0
\ne a, b \in k$ and $1 \le i \ne j \le n.$
 Hence  the degree sequence of $({\bar x}_1,
\ldots, {\bar x}_d)$ is a permutation of $\deg\mathrm{in}{\bar
x}_1, \ldots, \deg{\bar x}_d$  by Remark \ref {rm1.2u}. Since
$S/\mathfrak{p}$ is a domain, by Theorem \ref {th1.1a} we get
$$e({\bar x}_1, \ldots, {\bar x}_d; S/\mathfrak{p})= o({\bar
x}_1)\cdots o({\bar x}_d) e(S/\mathfrak{p}).$$ Note that $o({\bar
x}_i)= o(x_i)$ for all $1 \le i \le d$ and if consider
$S/\mathfrak{p}$ as a $S$-module then $$e({\bar x}_1, \ldots,
{\bar x}_d; S/\mathfrak{p})= e_S(x_1, \ldots, x_d;
S/\mathfrak{p}).$$ So $e_S(x_1, \ldots, x_d; S/\mathfrak{p}) =
o(x_1)\cdots o(x_d)e(S/\mathfrak{p}).$ Remember that
$$e(x_1,\ldots,x_d; S)
=\sum_{\mathfrak{p}\in\Lambda}\ell(S_{\mathfrak{p}})e_S(x_1,\ldots,x_d;
S/\mathfrak{p})$$ (see \cite[Theorem 11.2.4]{SH}), we obtain
$$e(x_1,\ldots,x_d; S)
= o(x_1)\cdots o(x_d)
\sum_{\mathfrak{p}\in\Lambda}\ell(S_{\mathfrak{p}})e(S/\mathfrak{p})=
o(x_1)\cdots o(x_d)e(S).$$ Thus $e(x_1,\ldots, x_d; S)=
o(x_1)\cdots o(x_d) e(S).$
\end{proof}

\vskip.2cmWe recall now the notion of weak-(FC)-sequences in
\cite{Vi}. This sequence is a kind of superficial sequences,  and
it has proven to be useful in several  contexts (see e.g.
\cite{DMT, DV, Fiber, VT}).
\begin{definition}[\cite{Vi}] \label{de1.2}
Let  $(A, \frak n)$  be  a  Noetherian   local ring with maximal
ideal $\frak{n}.$ Let $I_1,\ldots,I_s$ be ideals of $A$ such that
 $I= I_1\cdots I_s$ is non-nilpotent.
 An element $x \in
A$ is called a  {\it weak}-(FC)-{\it element} of $(I_1,\ldots,
I_s)$ if there exists $i \in \{ 1, \ldots, s\}$ such that $x \in
I_i$ and the following conditions are satisfied:
 \begin{enumerate}[(FC1):]
 \item For all large $n_i$
and for all $n_1,\ldots,n_{i-1},n_{i+1}, \ldots, n_s \geq 0,$
$$(x)\bigcap I_1^{n_1}\cdots I_i^{{n_i}+1}\cdots I_s^{n_s} =
xI_1^{n_1}\cdots I_i^{n_i}\cdots I_s^{n_s}.$$ \item $x$ is an
$I$-filter-regular element, i.e.,\;$0_A:x \subseteq 0_A:
I^{\infty}.$
 \end{enumerate}
        Let $x_1, \ldots, x_t$ be elements of $A$. For any $0\le i \le t,$
set      $A_i = \dfrac{A}{(x_1, \ldots, x_{i})}$.  Then $x_1,
\ldots, x_t$ is called a  {\it weak}-(FC)-{\it sequence}
 of
$(I_1,\ldots, I_s)$ if $\bar x_{i + 1}$ (the image of $x_{i + 1}$
in $A_i$) is a weak-(FC)-element  of $(I_1A_i,\ldots, I_sA_i)$ for
all $i = 0, \ldots, t-1$. If a weak-(FC)-sequence consists of
$k_1$ elements of $I_1,\ldots,k_s$ elements of $I_s$
($k_1,\ldots,k_s \ge 0$), then it is called  a weak-(FC)-sequence
of $(I_1,\ldots,I_s)$ of {\it the type} $(k_1,\ldots,k_s)$. A
weak-(FC)-sequence $x_1, \ldots, x_t$ is called a {\it maximal
weak-(FC)-sequence} if  $I\subseteq \sqrt{(x_1, \ldots, x_t)}.$
\end{definition}

   Let $\frak I$ be  $\frak n$-primary
and $I$  a non-nilpotent ideal of $A.$  By using Rees'lemma in
\cite{Re}, one (see e.g. \cite {Vi4, Vi2, Vi3,Fiber}) showed that
if $I$ is non-nilpotent, the length of any maximal
weak-(FC)-sequence in $I$ with respect to $(I, \frak I)$ is equal
to $\ell(I)=\ell$ and if
   $x_1,\ldots, x_{\ell}$ is a
weak-(FC)-sequence in $I$ with respect to $(I, \frak I),$ then
  $(x_1, \ldots,x_{\ell})$ is
a minimal reduction of $I.$ Moreover, as in the proof of Theorem
3.3 \cite {Vi4}, then any
 minimal reduction of $I$
is generated by a maximal weak-(FC)-sequence in $I$ with respect
to $(I, \frak I)$ (see e.g. \cite{VD, VT0, VT10}).

\vskip.2cmBefore giving the results in the rest of this section,
we need the following lemma on the relationship between  minimal
reductions and maximal weak-(FC)-sequences.

\begin{lemma}\label{lm1.4} Let $I$ be a non-nilpotent ideal of
$(A, \frak n)$ and $J$ a minimal reduction of $I.$  Then  $J$ is
generated by a  weak-$(FC)$-sequence of length $\ell(I)$  in $J$
of $(J, \frak n)$ and this sequence is also a  weak-(FC)-sequence
in $I$ of $(I, \frak n).$
 \end{lemma}
\begin{proof} Let $
x_1,\ldots,x_p$ be a maximal weak-(FC)-sequence in $J$ of $(J,I,
\frak n).$ Since $J$ is a
 reduction of $I,$  it easily follows that  $
x_1,\ldots,x_p$ is a weak-(FC)-sequence of $(J,\frak n),$ and $
x_1,\ldots,x_p$ is also a weak-(FC)-sequence of $(I, \frak n).$

By induction on $i \le p,$
 we    prove that  $$(x_1,\ldots, x_{i})
\bigcap J^{n+1}I^m = (x_1,\ldots, x_{i}) J^nI^m$$ for all large
$n$ and $m \ge 0.$ The  case of $i=0$ is trivial. Set $q =
(x_1,\ldots, x_{i-1})$ and $A' =A/q.$  Denote by $x_i'$ the image
of $x_i$ in $A'.$ Now suppose that the result has been proved for
$i-1 \ge 0.$  Since $x_i'$ in $JA'$  is a weak-(FC)-element of
$(JA', IA', \frak nA'),$ $$(x_i')\bigcap J^{n+1}I^mA' = x_i'J^n
I^mA'$$ for all large $n$ and $m \ge 0.$
 So
$[J^{n+1}I^m+q]\bigcap [(x_i) +q] = x_iJ^nI^m + q$ for all large
$n$ and $m \ge 0.$ Hence
$$J^{n+1}I^m\bigcap [(x_i) +q] = x_iJ^nI^m  + q \bigcap J^{n+1}I^m$$
for all large $n$ and $m \ge 0.$ By the inductive assumption, $q
\bigcap J^{n+1}I^m= qJ^{n}I^m$ for all large $n$ and $m \ge 0.$
Hence
$$(x_1,\ldots, x_{i})\bigcap J^{n+1}I^m = (x_1,\ldots, x_{i})J^{n}I^m$$
for all large $n$ and $m \ge 0.$ The induction is complete.
Consequently,
$$(x_1,\ldots, x_{p})\bigcap J^{n+1}I^m = (x_1,\ldots, x_{p})J^{n}I^m$$
for all large $n$ and $m \ge 0.$ Since  $x_1,\ldots, x_p$  is a
maximal weak-(FC)-sequence in $J$ of $(J, I, \frak n),$ $
J^{n+1}I^m \subset (x_1,\ldots, x_{p})$ for all large $n$ and $m
\ge 0.$ Thus
$$J^{n+1}I^m = (x_1,\ldots, x_{p})J^{n}I^m$$ for all large $n$ and
$m \ge 0.$ From this it follows that $(x_1,\ldots, x_{p})$ is a
reduction of both $J$ and $I.$ So $p \ge \ell(I).$ Note that
$\ell(JI) = \ell(I),$ and moreover $p \le \ell(JI)$ by
\cite[Theorem 3.4(iv)]{Vi2}, we get $p = \ell(I).$ Hence
$(x_1,\ldots, x_{p})$ is a minimal reduction of both $J$ and $I.$
So $J= (x_1,\ldots, x_{p}).$ We get the proof of the lemma.
\end{proof}

\vskip.2cm Now we will apply the above results for mixed
multiplicities of ideals. Risler and Teissier in 1973 \cite{Te}
defined  mixed multiplicities of  ideals of dimension $0$  and
interpreted them as the Hilbert-Samuel multiplicity of ideals
generated by general elements. Katz and Verma in 1989 \cite{KV}
started the investigation of mixed multiplicities of ideals of
positive height. The case of arbitrary ideals, Viet in 2000
\cite{Vi} described  mixed multiplicities as the Hilbert-Samuel
multiplicity via (FC)-sequences.  In past years, the relationship
between mixed multiplicities  and the Hilbert-Samuel multiplicity,
and the properties similar to that of the Hilbert-Samuel
multiplicity have attracted much attention (see e.g. \cite{CP, DV,
HHRT,SH,KV, KR1, KR2, KT, MV,  Ro, Sw, Ve, Vi4, Vi2, Vi3,VD1, VDT,
VM, VT0, VT1}).

\vskip.2cmLet  $(A, \frak n)$  be  a  Noetherian   local ring with
the maximal ideal $\frak{n}.$ Let $I_1,\ldots,I_s$ be ideals of
$A$ such that  $I= I_1\cdots I_s$ is non-nilpotent.  Set $\dim
A/0_A:I^\infty = q.$ Let $J$ be a $\frak n$-primary ideal. Put
 ${\bf 1} = (1,\ldots,
1) \in \mathbb{N}^s;$ $\mathrm{\bf k}!= k_1!\cdots k_s!;$\;
$|\mathrm{\bf k}| = k_1+\cdots+k_s;$ $\mathrm{\bf n}^\mathrm{\bf
k}= n_1^{k_1}\cdots n_s^{k_s}$ for each $\mathrm{\bf
n}=(n_1,\ldots,n_s); \mathrm{\bf k}=(k_1,\ldots,k_s)\in
\mathbb{N}^s$ and $\mathrm{\bf n} \ge {\bf 1}.$ Moreover, set
$$\mathrm{\bf I}= I_1,\ldots,I_s;\;\mathrm{\bf I}^{[\mathrm{\bf
k}]}=I_1^{[k_1]},
 \ldots,I_s^{[k_s]};\;
 \mathrm{\bf I}^{\bf n}= I_1^{n_1}\cdots I_s^{n_s}.$$
 Remember  that by
\cite[Proposition 3.1]{Vi},
 $\ell_A\bigg[\dfrac{J^{n_0}\mathrm{\bf I}^{\bf n}}{J^{n_0+1}\mathrm{\bf I}^{\bf n}}\bigg]$
is a polynomial of degree $(q-1)$ for all large $n_0,{\bf n}.$ The
terms of total degree $(q-1)$ in this polynomial have the form $$
\sum_{k_0\:+\mid\mathrm{\bf k}\mid\;=\;q-1}e(J^{[k_0+1]},
\mathrm{\bf I}^{[\mathrm{\bf k}]}; A)\dfrac{n_0^{k_0} \mathrm{\bf
n}^\mathrm{\bf k}}{k_0!\mathrm{\bf k}!}.$$ Then recall that
$e(J^{[k_0+1]},\mathrm{\bf I}^{[\mathrm{\bf k}]}; A)$ is called
the {\it  mixed multiplicity of $J,\mathrm{\bf I}$ of the type
$(k_0, \mathrm{\bf k})$}(see e,g.\cite {HHRT, Ve}).

\vskip.2cm For  statements of the next result, we would like to
recall the following fact.

\begin{remark} \label{rm36a} Let $I_1,\ldots, I_s$ be
   ideals in the graded ring $S$ with
    $\mathrm{ht}(I_1\cdots I_s)  = h >0.$ Let $k_0, \mathrm{\bf k}$
    be
   non-negative integers such that $k_0+ |\mathrm{\bf k}| = \dim S -1$ and
   $ |\mathrm{\bf k}| < h.$
   Then since $t= |\mathrm{\bf k}| < h,$
by \cite[Theorem 3.5 (ii)]{Vi} or \cite[Proposition 3.1
(vii)]{Vi2} (see e.g. \cite{DMT, DV, VD1, VT}), there exists a
weak-(FC)-sequence $x_1, \ldots, x_t$  of $(\mathrm{\bf I}, \frak
m)$ of the type $(\mathrm{\bf k}, 0),$ and in this case, $\dim
S/(x_1, \ldots, x_t) = d-t$ and
$$e(\frak m^{[k_0+1]}, \mathrm{\bf I}^{[\mathrm{\bf k}]}; S) =
e(S/(x_1, \ldots, x_t)).$$ Since $\dim S/(x_1, \ldots, x_t) =
d-t,$   $x_1, \ldots,x_t$ is a subsystem of parameters of $S.$
\end{remark}

Then as an application of Theorem \ref{th1.1v} and known results
in Remark \ref{rm36a}  on mixed multiplicities, we get the
following theorem.

\begin{theorem}\label{th1.3} Let $\mathrm{\bf I}=I_1,\ldots, I_s$ be
   ideals of $S$ with
    $\mathrm{ht}(I_1\cdots I_s)  = h >0.$ Let $k_0, \mathrm{\bf k}$
    be
   non-negative integers such that $k_0+ |\mathrm{\bf k}| = \dim S -1$ and
   $ |\mathrm{\bf k}| < h.$ Let $x_1, \ldots, x_t$ be
 a weak-$(FC)$-sequence   of $(\mathrm{\bf I},
\frak m)$ of the type $(\mathrm{\bf k}, 0).$
  Assume that $\mathrm{in}
x_1,\ldots,\mathrm{in} x_t$ is a subsystem of parameters for $S.$
 Then we have
 $$e(\frak m^{[k_0+1]}, \mathrm{\bf I}^{[\mathrm{\bf k}]}; S)
= o(x_1)\cdots o(x_t)e(S).$$
\end{theorem}
\begin{proof} Since $t= |\mathrm{\bf k}| < h$ and $x_1, \ldots, x_t$
is
 a weak-$(FC)$-sequence   of $(\mathrm{\bf I},
\frak m)$ of the type $(\mathrm{\bf k}, 0),$ by Remark
\ref{rm36a}, we get $\dim S/(x_1, \ldots, x_t) = d-t$ and
$$e(\frak m^{[k_0+1]}, \mathrm{\bf I}^{[\mathrm{\bf k}]}; S) =
e(S/(x_1, \ldots, x_t)).$$   Now since $\mathrm{in}
x_1,\ldots,\mathrm{in} x_t$ is also a subsystem of parameters for
$S,$ $$\dim S/(\mathrm{in}x_1, \ldots, \mathrm{in}x_t) = d-t.$$
Hence there exist homogeneous elements of degree $1$:
$x_{t+1},\ldots, x_d \in \frak m $  such that $$(x_{t+1},\ldots,
x_d)(S/(x_1, \ldots, x_t)) $$ and $(x_{t+1},\ldots,
x_d)(S/(\mathrm{in}x_1, \ldots, \mathrm{in}x_t)) $ are  minimal
reductions of $\frak m(S/(x_1, \ldots, x_t))$ and $\frak
m(S/(\mathrm{in}x_1, \ldots, \mathrm{in}x_t)),$ respectively by
\cite{NR}. Then $$x_1, \ldots,x_t, x_{t+1},\ldots, x_d$$ and
$\mathrm{in}x_1, \ldots,\mathrm{in}x_t, \mathrm{in}x_{t+1},\ldots,
\mathrm{in}x_d$ are  systems of parameters of $S.$ Hence by
Theorem \ref{th1.1v}, we have \begin{align*}e(x_1, \ldots,x_t,
x_{t+1},\ldots, x_d; S) &= \deg x_{t+1}\cdots \deg x_do(x_1)
\cdots o(x_t)e(S)\\&= o(x_1) \cdots o(x_t) e(S).\end{align*}
 Recall
that $x_1, \ldots,x_t$ is a weak-(FC)-sequence, $x_1, \ldots,x_t$
is an $I$-filter-regular sequence and $t < \mathrm{ht}I,$  here
$I=I_1\cdots I_s.$ Hence it follows  by \cite{AB} that
 $$e(x_1, \ldots,x_t, x_{t+1},\ldots,
x_d; S) = e(x_{t+1},\dots,x_d; S/(x_1,\ldots,x_{t})).$$ Since
$e(x_{t+1},\dots,x_d; S/(x_1,\ldots,x_{t})) = e(\frak m;
S/(x_1,\ldots,x_{t})) = e( S/(x_1,\ldots,x_{t}))$ by \cite[Theorem
1]{NR}, we get $e( S/(x_1,\ldots,x_{t}))= e(x_1, \ldots,x_t,
x_{t+1},\ldots, x_d; S).$ Thus $$e(\frak m^{[k_0+1]}, \mathrm{\bf
I}^{[\mathrm{\bf k}]}; S) = o(x_1)\cdots o(x_t)e(S).$$
 \end{proof}

 Denote by $ R(I) =
  \bigoplus_{n\ge  0}I^n$
  the Rees algebra of an ideal $I.$ Then as an application of  Theorem
  \ref{th1.3}, we obtain the following.

\begin{corollary}\label{th3.6a} Let $I$ be an  equimultiple
ideal  of $\mathrm{ht}I =h > 0.$  Set $\dim S =d.$ Let  $x_1,
\ldots, x_h$ be a weak-$(FC)$-sequence in $I$ of $(I, \frak m).$
Assume that $\mathrm{in} x_1,\ldots,\mathrm{in} x_h$ is a
subsystem of parameters for $S.$
 Then  we have
\begin{enumerate}[\rm (i)]
 \item $ e({\frak m}^{[d-i]},I^{[i]}; S) = o(x_1) \cdots o(x_i)e(S)$  for all $ i \le
 h-1.$
 \item $e(R(I)) =
(1+\sum_{i=1}^{h-1}o(x_1) \cdots o(x_i))e(S).$
\end{enumerate}
\end{corollary}
\begin{proof} Since  $x_1, \ldots, x_i$ is a weak-$(FC)$-sequence in $I$ of $(I,
\frak m)$ and $\mathrm{in} x_1,\ldots,\mathrm{in} x_i$ is a
subsystem of parameters for $S$ for all $i \le h-1.$
 By Theorem \ref {th1.3}, we get $$ e({\frak
m}^{[d-i]},I^{[i]}; S) = o(x_1) \cdots o(x_i)e(S)$$  for all $ i
\le
 h-1.$ We have (i).  Since $I$ is an equimultiple ideal  of
 ht$I= h,$  for any $i \ge h,$ we have $ e({\frak m}^{[d-i]},I^{[i]}; S) = 0$
 by \cite{V1} (see e.g. \cite{Vi4, VD}). So by (i) and \cite[Theorem 3.1]{V1}  (see e.g.   \cite [Theorem
3.2] {Vi4} or  \cite [Corollary 4.7 (ii)]{VD}), we get (ii).
\end{proof}

\section{Mixed Multiplicities of Homogeneous Ideals}

As applications of Section 2, in this section, we  determine
formulas for  mixed multiplicities and multiplicities
 of Rees rings  of homogeneous  equimultiple
ideals in $S.$

\vskip.2cm Let $I$ be a homogeneous ideal of $S.$  An ideal $J$ is
called a {\it homogeneous minimal reduction } of $I$ if $J$ is a
minimal reduction of $I$ and $J$ is homogeneous.
 $I$ is called a {\it homogeneous  equimultiple ideal} if there
exists a homogeneous minimal reduction $J$ of $I$ generated by
$\mathrm{ht}I$ homogeneous elements.

\vskip.2cm
 Now, we  investigate the relationship between the degree sequence of  minimal
 reductions of
homogeneous equimultiple ideals and the degree
 sequence of weak-(FC)-sequences  that
will be used in the proof for  Corollary  \ref {th3.6} of this
section.

\begin{lemma} \label{lm1.6} Let $I$ be a homogeneous  equimultiple
ideal  of $\mathrm{ht}I =h > 0.$   Let $J$ be a homogeneous
minimal reduction of $I$  with the degree sequence
$(c_1,\ldots,c_{h}).$
 Then $J$ is generated by  a
weak-$(FC)$-sequence  $ x_1,\ldots,x_{h}$ in $I$ of $(I, \frak m)$
such that $\mathrm{in} x_1,\ldots,\mathrm{in} x_h$ is a subsystem
of parameters for $S$ and
 $$(o(x_1),\ldots,o(x_h))= (c_1,\ldots,c_{h}).$$
  \end{lemma}
\begin{proof} By Proposition \ref {lm1.01},
we can assume that $y_1, \ldots, y_h$ is a homogeneous minimal
basis of $J$ such that $y^*_1, \ldots, y^*_h$ is a minimal basis
of $J^*,$ and the degree sequences of $J$ and $J^*$ are $(\deg
y_1,\ldots, \deg y_h) = (c_1,\ldots,c_{h}).$
 By Lemma \ref{lm1.4}, $J$  is generated by a weak-(FC)-sequence
$x_1,\ldots, x_{h}$ in $J$ of $(J, \frak m)$ and this sequence is
also a weak-(FC)-sequence  in $I$ of $(I, \frak m).$ Then notice
that the sequence $x_1,\ldots,x_i, (x_{i+1}-\sum_{j=1}^i
a_jx_j),\ldots, x_{h}$ and the sequence $u_1x_1,\ldots,
u_{h}x_{h}$
 are also  weak-(FC)-sequences in $J$ of $(J, \frak
m)$ and $(u_1x_1,\ldots, u_{i}x_{i})=(x_1,\ldots, x_{i})$ for all
$ 0 \ne u_1,\ldots,u_{h} \in k;$ $a_1,\ldots,a_i \in k$ and $i \le
h.$ Hence
  without loss of generality, we can assume
that   $x_1, \ldots, x_{h}$ are written in the following forms
$$ x_j=
y_{n_j}+a_{j(j+1)}y_{n_{j+1}}+\cdots+
    a_{jh}y_{n_{h}},$$
 here $a_{ij} \in k$,  $y_{n_1}, \ldots, y_{n_h}$ is a permutation
of $y_1, \ldots, y_h,$ moreover $o(x_j) = \deg y_{n_j};$ $x_j^* =
y^*_{n_j}$ for all $j = 1, \ldots, h.$  Set $U_i
=(x_{1},\ldots,x_{i})$ for all $i \le h.$ Since
$x^*_{1},\ldots,x^*_{i}$ are pairwise
 distinct, it follows that
$x^*_{1},\ldots,x^*_{i}$ and $\mathrm{in} x_1,\ldots,\mathrm{in}
x_i$ are minimal bases of $U^*_i$ and $\mathrm{in}U_i,$
respectively by Remark \ref{rm1.2u}.
 In this case,
$\mathrm{in} x_1,\ldots,\mathrm{in} x_h$ is a minimal basis of
$\mathrm{in}U_h.$  Because  $U_h = J$ is a homogeneous ideal  of
$\mathrm{ht}I =h > 0,$ $J =(\mathrm{in} x_1,\ldots,\mathrm{in}
x_h).$ So $\mathrm{in} x_1,\ldots,\mathrm{in} x_h$ is a subsystem
of parameters for $S.$

The following we will prove that $(o(x_1),\ldots,o(x_h))=
(c_1,\ldots,c_{h}).$

 Set $J_i =
(y_{n_1},\ldots,y_{n_i});$ $J_0 = U_0 = \{0\};$
     $S'_i = S/U_i;$ $S_i = S/J_i$ for all $0\le i \le h-1.$ Denote by
     $x'; \bar x,$  $J_i'; {\bar
J}_i$ and $\frak m_i'; \bar{\frak m}_i$ the images of $ x \in S,$
$ J$ and $\frak m$ in $S'_i$ and $S_i,$ respectively.  Since
$y^*_1, \ldots, y^*_h$ is a minimal basis of $J^*,$ it follows
that $y^*_{n_1},\ldots,y^*_{n_i}$ is a minimal basis of $J_i^*.$
So $J_i^* =(y^*_{n_1},\ldots,y^*_{n_i}).$ Note that $ x'_{{i+1}}$
is a weak-(FC)-element in $J_i'$ of $(J_i', \frak m_i'),$ we have
$$(x'_{{i+1}})\bigcap \frak m_i'^mJ_i'^{n} = x'_{{i+1}}\frak
m_i'^mJ_i'^{n-1}$$ for all large $m, n.$
 Since $x_{{i+1}}\frak m^{m -o(x_{{i+1}})} \subset
\frak m^m$ for all $m \ge o(x_{i+1}),$ it follows that
$x'_{{i+1}}\frak m_i'^{m -o(x_{{i+1}})}J_i'^n \subset \frak
m_i'^mJ_i'^n $ for all $m \ge o(x_{i+1}).$  Consequently, we
obtain
$$x'_{{i+1}}\frak m_i'^{m -o(x_{{i+1}})}J_i'^n \subset (x'_{{i+1}})
\bigcap \frak m_i'^mJ_i'^{n} = x'_{{i+1}}\frak m_i'^mJ_i'^{n-1}$$
for all large $m, n.$ Hence $x_{{i+1}}\frak m^{m -o(x_{{i+1}})}J^n
+ U_i \subset x_{{i+1}}\frak m^mJ^{n-1} + U_i$ for all large $m,
n.$ Therefore  we get \begin{equation}\label{eq3}o((x_{{i+1}}\frak
m^{m -o(x_{{i+1}})}J^n + U_i)^*/  U^*_i) \ge o((x_{{i+1}}\frak
m^mJ^{n-1} + U_i)^*/U^*_i)\end{equation} for all large $m, n.$
Emphasize that $U_i^* = (x^*_1,\ldots,x^*_i) =
(y^*_{n_1},\ldots,y^*_{n_i}) = J_i^*$ for all $1 \le i \le h-1;$
$U_0^* = J_0^* = \{0\}$ and $x_i^* = y^*_{n_i}$ for all $i = 1,
\ldots, h,$ first of all we have
\begin{equation}\label{eq4}(x_{{i+1}}\frak m^mJ^{n-1} +
U_i)^*/U_i^*= (y_{n_{i+1}}\frak m^mJ^{n-1} + J_i)^*/
J_i^*.\end{equation} Next, we need to show that
\begin{equation}\label{eq5} o\big(y_{n_{i+1}}\frak m^m J^{n-1}+ J_i/ J_i\big)=
o\big((y_{n_{i+1}}\frak m^mJ^{n-1} + J_i)^*/
J_i^*\big).\end{equation} Indeed, the  case of $i=0$ is true by
Proposition \ref {lm1.01} (iii). For $1\le i \le h-1,$
 since
$y_{n_1},\ldots,y_{n_i}$ is a part of a minimal basis of $J$  and
$$(y_{n_1},\ldots,y_{n_i})= J_i \subset y_{n_{i+1}}\frak m^mJ^{n-1}
+ J_i \subset J,$$ it follows that $y_{n_1},\ldots,y_{n_i}$ is a
part of a minimal basis of $y_{n_{i+1}}\frak m^mJ^{n-1} + J_i.$
Moreover $i < \mathrm{ht} J$ and $y^*_{n_1},\ldots,y^*_{n_i}$ are
pairwise distinct,
 there exist $j$ elements for a
certain integer $j
> 0$
$$f_1,\ldots, f_j \in y_{n_{i+1}}\frak m^m J^{n-1}+
J_i$$ such that $f_1,\ldots, f_j,y_{n_1},\ldots,y_{n_i}$ is a
minimal basis of $y_{n_{i+1}}\frak m^mJ^{n-1} + J_i$ and
$$f^*_1,\ldots, f^*_j,y^*_{n_1},\ldots,y^*_{n_i}$$
are pairwise distinct (see  the proof of Proposition \ref
{lm1.01}). Therefore $f^*_1,\ldots,
f^*_j,y^*_{n_1},\ldots,y^*_{n_i}$ is a minimal basis of
$(y_{n_{i+1}}\frak m^m J^{n-1}+ J_i)^*$  by Remark \ref{rm1.2u}.
 Then it is easily seen that
$$o\big(y_{n_{i+1}}\frak m^m J^{n-1}+
J_i/ J_i\big) = o(\{f_1,\ldots, f_j\})$$ and
$o\big((y_{n_{i+1}}\frak m^mJ^{n-1} + J_i)^*/ J_i^*\big) =
o(\{f^*_1,\ldots, f^*_j\}).$  Consequently  we have (\ref{eq5}).

 Note that
$y_{n_{i+1}}\frak m^m J^{n-1}+ J_i/ J_i = \bar
y_{n_{i+1}}\bar{\frak m}_i^m{\bar J}_i^{n-1},$ hence we obtain
$$o\big((x_{{i+1}}\frak m^m J^{n-1}+ U_i)^*/U_i^*\big)= o( \bar
y_{n_{i+1}}\bar{\frak m}_i^m{\bar J}_i^{n-1})$$ by (\ref{eq4}) and
(\ref{eq5}).
 And likewise,
$$o\big((x_{i+1}\frak m^{m-o(x_{i+1})}J^n + U_i)^*/U_i^*\big) =
o(\bar y_{n_{i+1}}\bar{\frak m}_i^{m -o(x_{i+1})}{\bar J}_i^n).$$
Consequently, $$o(\bar y_{n_{i+1}}\bar{\frak m}_i^{m
-o(x_{i+1})}{\bar J}_i^n) \ge o(\bar y_{n_{i+1}}\bar{\frak
m}_i^m{\bar J}_i^{n-1})$$ for all large $m, n$ by (\ref{eq3}).
Recall that $o(x_i) = \deg y_{n_i}$ for all $1 \le i \le h.$
Moreover, we have $\deg y_{n_j} = \deg \bar y_{n_j}$ for all $i+1
\le j \le h.$ Since $I$ is a homogeneous equimultiple ideal,
$y_1,\ldots, y_{h}$ is a subsystem of homogeneous parameters. From
this it follows that $\bar J_i$ is generated by a subsystem of
homogeneous parameters in $S_i,$ and $\bar y_{n_{i+1}} \in \bar
J_i$ is an element of homogeneous parameters.
 Because  a subsystem of
parameters for $S_i$ is algebraically independent over $k$ (see
e.g. \cite [Corollary 11.21]{A}),
  we get  $$o(\bar y_{n_{i+1}}\bar{\frak m}_i^m {\bar J}_i^n) = \deg\bar
y_{n_{i+1}}+ m + n.o({\bar J}_i)$$ for all $m, n
> 0.$ Therefore it can easily be seen that
$$m+n.o({\bar J}_i) = o(\bar y_{n_{i+1}}\bar{\frak m}_i^{m
-o(x_{i+1})}{\bar J}_i^n) \ge o(\bar y_{n_{i+1}}\bar{\frak
m}_i^m{\bar J_i}^{n-1})= m+\deg\bar y_{n_{i+1}}+(n-1)o({\bar
J}_i)$$ for all large $m, n.$ So $ o(\bar J_i) \ge \deg\bar
y_{n_{i+1}}.$ Note that we always have $ o(\bar J_i) \le \deg\bar
y_{n_{i+1}},$ consequently
 $ o(\bar J_i) = \deg\bar
y_{n_{i+1}}.$ Remember that $$\bar J_i= J/J_i= (\bar
y_{n_{i+1}},\ldots, \bar y_{n_h}) $$ and $\deg\bar y_{n_{i+1}} =
\deg y_{n_{i+1}}$  for all $0 \le i \le h-1,$ hence
  $o(\{y_{n_{i+1}},\ldots, y_{n_h} \})  = \deg
y_{n_{i+1}}$ for all  $0 \le i \le h-1.$  From this it follows
that $\deg y_{n_1}\leq\cdots \leq\deg y_{n_h}.$
   So $$(\deg y_{n_1},\ldots, \deg y_{n_h})$$ is the degree
   sequence of $J,$ i.e.,  $(\deg y_{n_1},\ldots, \deg y_{n_h})= (c_1,\ldots,c_{h}).$
Then we get $(o(x_1),\ldots,o(x_h))= (c_1,\ldots,c_{h})$ because
$o(x_i) = \deg y_{n_i}$ for all $1 \le i \le h.$
\end{proof}

\vskip.2cm The following theorem  answers to the question (see
Question 1) on determining mixed multiplicities and multiplicities
of Rees rings of homogeneous
 ideals generated by arbitrary subsystems of parameters.

\begin{corollary}\label{th3.6} Let $I$ be a homogeneous  equimultiple
ideal in $S$  of $\mathrm{ht}I =h > 0.$ Let $(a_1, \ldots,a_h)$ be
the degree sequence of a homogeneous minimal reduction of $I.$ Set
$\dim S =d.$ Then we have:
\begin{enumerate}[\rm (i)]
 \item $ e({\frak m}^{[d-i]},I^{[i]} ; S) = a_1 \cdots a_ie(S)$  for all $ i \le
 h-1.$
 \item $e(R(I)) =
(1+\sum_{i=1}^{h-1}a_1 \cdots a_i)e(S).$
\end{enumerate}
\end{corollary}
\begin{proof} Let  $J$ be a homogeneous
minimal reduction of $I$  generated by homogeneous elements of the
degree sequence $(a_1, \ldots,a_h).$ By Lemma \ref {lm1.6}, there
exists a weak-(FC)-sequence $ x_1,\ldots,x_{h}$ in $I$ of $(I,
\frak m)$ such that $\mathrm{in} x_1,\ldots,\mathrm{in} x_h$ is a
subsystem of parameters for $S$ and
 $(o(x_1),\ldots,o(x_h))= (a_1,\ldots,a_{h}).$
   Hence by Corollary  \ref{th3.6a}, we get the proof
of this corollary.
\end{proof}

\begin{remark} \label{rm36b} Corollary \ref{th3.6} is also Theorem
3.1 given by Viet in \cite{ViS}. Recall that the proof of
\cite{ViS} is based on the existence of homogeneous
weak-(FC)-sequences in $I$ as in \cite [Note 1]{ViS}. However,
this is not correct because in general there does not exist a
homogeneous weak-(FC)-sequence $ x_1,\ldots,x_h$ as the statements
in Lemma \ref{lm1.6} by the Remark of Lemma 1.4 \cite {Tr}. So by
another approach, the paper showed that Theorem 3.1 and hence the
other results on mixed multiplicities and multiplicities of Rees
rings in \cite{ViS} are true. So  Question 1 has a positive
answer.
\end{remark}

Theorem \ref{th3.6} covered  \cite [Theorem 3.3]{Tr} and \cite
[Theorem 3.6]{H} that \cite {H,Tr} studied this problem in the
case that $I$ is a homogeneous ideal generated by a subsystem of
parameters $x_1, \ldots,x_h$ which is an $I$-filer-regular
sequence with $\deg x_1 \le \cdots \le \deg x_h.$

\vskip.2cm
 Moreover, the
following result will show the decision of the degree sequence for
mixed multiplicities and multiplicities of Rees rings of
homogeneous  equimultiple ideals in graded rings.

\begin{corollary}\label{co37c} Let $I$ and $E$
be homogeneous equimultiple ideals of $S.$ Suppose that
homogeneous minimal reductions of $I$ and $E$  have the same
degree sequence. Then $e(R(I)) = e(R(E))$ and  $ e({\frak
m}^{[d-i]},I^{[i]}; S) = e({\frak m}^{[d-i]},E^{[i]}; S)$
  for all $ i \le h-1.$
\end{corollary}

\begin{proof} Let $(x_1,
\ldots,x_h)$ and $(y_1, \ldots,y_h)$ be  homogeneous minimal
reductions of $I$ and $E$ with the  degree sequences $\deg x_1 \le
\cdots \le \deg x_h$ and $\deg y_1 \le \cdots \le \deg y_h,$
respectively. From the assumption we deduce that  $\deg x_i = \deg
y_i$ for all $i \le h.$ So  by Corollary \ref {th3.6} we get $
e({\frak m}^{[d-i]},I^{[i]}; S) = e({\frak m}^{[d-i]},E^{[i]}; S)$
  for all $ i \le h-1$ and $e(R(I)) = e(R(E)).$
\end{proof}
\begin{remark} \label{rm36f} Note that for any  equimultiple ideal $I$ of
 ht$I= h,$  for any $i \ge h,$ we have $ e({\frak m}^{[d-i]},I^{[i]}; S) = 0$
 by \cite{V1}
(see e.g. \cite{Vi4, VD}). Moreover, by Corollary \ref{th3.6}, it
follows that if $(a_1, \ldots,a_h)$ and $(b_1, \ldots,b_h)$ are
the degree sequences of homogeneous minimal reductions of
homogeneous equimultiple ideals $I$ and $E,$ respectively, then $
 e({\frak m}^{[d-i]},I^{[i]}; S) = e({\frak
m}^{[d-i]},E^{[i]}; S)$ for all $ i$    if and only if $a_i = b_i$
for all $ i \le h-1.$
\end{remark}

\end{document}